\documentclass{amsart}
\usepackage{amsmath, amssymb, hyperref}

\input xy
\xyoption{all}

\newcommand{\ran}{\mathop{\mathrm{ran}}}

\newcommand{\free}[1]{\underset{\scriptscriptstyle #1}{\displaystyle{\ast}}\,}

\title[Maximal C$^*$-covers]{Identification of maximal C$^*$-covers of some operator algebras}

\author{Benton L.\ Duncan}
\address{
Department of Mathematics \\
Illinois State University \\
Normal, Illinois \\
USA }
\email{bldunc1@ilstu.edu}



\begin{document}

\theoremstyle{plain}
\newtheorem{theorem}{Theorem}[section]
\newtheorem{lemma}[theorem]{Lemma}
\newtheorem{proposition}[theorem]{Proposition}
\newtheorem{corollary}[theorem]{Corollary}

\theoremstyle{definition}
\newtheorem{definition}[theorem]{Definition}
\newtheorem*{algorithm}{Algorithm}
\newtheorem*{construction}{Construction}
\newtheorem{example}[theorem]{Example}

\theoremstyle{remark}
\newtheorem*{conjecture}{Conjecture}
\newtheorem*{acknowledgement}{Acknowledgements}
\newtheorem{remark}[theorem]{Remark}

\begin{abstract} We use results on inclusions of free products and extensions of completely positive maps to determine the maximal C$^*$-envelope for upper triangular $3 \times 3$ matrices. We consider these same results in the context of larger upper triangular matrices and graph algebras associated to cycle graphs. 
\end{abstract} 

\maketitle

\section{Introduction}

In 1999, Blecher \cite{Blecher} introduced the concept of the maximal C$^*$-cover of an operator algebra. This algebra encodes the completely contractive representation theory of an operator algebra. In particular, given an operator algebra $A$, the maximal C$^*$-cover C$^*_{\rm max} (A)$ is the C$^*$-algebra generated by a completely isometric copy of $A$ so that any completely contractive representation of $A$ into a C$^*$-algebra extends to a $^*$-representation of C$^*_{\rm max} (A)$ into the same C$^*$-algebra. In that paper, Blecher provided a concrete example for the operator algebra $T_2$ of the upper triangular $2 \times 2$ matrices which is the universal C$^*$-algebra generated by a contractive nilpotent of degree 2 (see the first row of Table 1 in \cite{loring}). 

Several authors have extended the use of the maximal C$^*$-cover for operator algebras. It has found use in \cite{BlecherDuncan} to understand a version of $Ext$ for nonselfadjoint operator algebras. More recently, in \cite{ClouatreRamsey} and \cite{Thompson} the maximal C$^*$-cover of an operator algebra was used to study a notion of residually finite dimensional nonselfadjoint operator algebras. At the same time, the use of the maximal C$^*$-covers is constrained somewhat by the lack of concrete examples of maximal C$^*$-covers. The author made a general attempt to understand the maximal C$^*$-cover of an operator algebra using Banach $*$-algebra methods in \cite{Duncan2}, but even there concrete examples were absent.

In a similar spirit to that of this paper, in \cite{KirchbergWassermann}, the authors consider the universal C$^*$-algebra generated by an operator system (arising in a similar manner from universal properties) and there they construct maximal C$^*$-algebras of the operator systems $\mathbb{C}, \mathbb{C} \oplus \mathbb{C},$ and $ \mathbb{C} \oplus \mathbb{C} \oplus \mathbb{C}$. They also show that the universal C$^*$-algebra of the operator system $M_2(\mathbb{C})$ is non-exact. An important difference to note is that viewing these as operator systems is different from viewing them as operator algebras, since in all three cases, as C$^*$-algebras, their (algebraic) maximal C$^*$-cover is themselves. In that context, maximal C$^*$-envelopes of operator systems were used by Davidson and Kennedy in a recent preprint \cite{DavidsonKennedy}. 

In this paper, we do two things. First, we build tools to more readily understand inclusions of free products of C$^*$-algebras. These extend results of \cite{Pedersen} concerning pushout diagrams and their extensions in \cite{ArmstrongDykemaExelLi} to general inclusions of C$^*$-algebras in free products. We also extend recent results on conditional expectations of free products from \cite{DavidsonKakariadis}. Second, we use the tools we develop to consider a method to build concrete maximal C$^*$-covers in two contexts. There we will see that even simple operator algebras give rise to non-exact maximal C$^*$-covers. We will also demonstrate how quickly the complications of this technique interfere in coherent natural generalizations of the example we build.

\section{Inclusions of free products}

Before we focus on identifying the maximal C$^*$-covers of interest, we need to provide a context for placing them into a concrete C$^*$-algebra. To do this we will need to extend some well-known results on C$^*$-algebras. We are able to do this because we will be identifying our operator algebras as free products of well understood operator algebras. These extensions of known results are interesting in their own right. 

We start by letting $A_{j,k}$ be the C$^*$-subalgebra of $M_k \otimes C[0,1]$ of the form 
\[ \begin{bmatrix} 
C[0,1]  & \cdots & 0 & 0 & 0 & 0 & \cdots  & 0 \\ 
\vdots  & \cdots & \ddots & \ddots & \ddots & \ddots & \cdots &  \vdots \\
0  & \cdots & C[0,1] & 0 & 0 & 0 & \cdots  & 0 \\ 
0  & \cdots & 0 & C[0,1] & C[0,1] & 0 & \cdots  & 0 \\ 
0  & \cdots & 0 & C[0,1] & C[0,1] & 0 & \cdots  & 0 \\ 
0  & \cdots & 0 & 0 & 0 & C[0,1] & \cdots  & 0 \\ 
\vdots  & \cdots & \ddots & \ddots & \ddots & \ddots & \cdots  & \vdots \\
0  & \cdots & 0 & 0 & 0 & 0 & \cdots  & C[0,1] \end{bmatrix}. \]
Here the first row with two copies of $C[0,1]$ is the $j$-th row of the matrix. Of course, this only make sense if $j \leq k-1$. 

\begin{example}
For $k= 2$ there is only one possible algebra \[ A_{1,2} = \begin{bmatrix} C[0,1] & C[0,1] \\ C[0,1] & C[0,1]\end{bmatrix} \] but when $k=3$ there are two options
\begin{align*}
A_{1,3} & = \begin{bmatrix} C[0,1] & C[0,1] & 0 \\ C[0,1] & C[0,1] & 0 \\ 0 & 0 & C[0,1] \end{bmatrix} \mbox{ and } \\
A_{2,3} & = \begin{bmatrix} C[0,1] & 0 & 0 \\ 0 & C[0,1] & C[0,1] \\ 0 & C[0,1] & C[0,1] \end{bmatrix}
\end{align*}
\end{example}

Next we write $A_{k,k}$ for the algebra
\[ \begin{bmatrix} 
C[0,1] & 0 & 0 & \cdots & 0 & 0 & C[0,1] \\
0 & C[0,1] & 0 & \cdots & 0 & 0 & 0 \\
0 & 0 & C[0,1] & \ddots & 0 & 0 & 0 \\
\vdots & \vdots & \vdots & \ddots & \vdots & \vdots & \vdots \\
0 & 0 & 0 & \cdots & C[0,1] & 0 & 0 \\
0 & 0 & 0 & \cdots & 0 & C[0,1] & 0 \\
C[0,1] & 0 & 0 & \cdots & 0 & 0 & C[0,1] \end{bmatrix}. \]

We denote by $e_{i,j}$ the elementary matrix in $M_k\otimes C[0,1]$ with $1$ in the $i$-$j$ position and zero everywhere else, and we let 
$D_k$ be the diagonal subalgebra of $A_{j,k}$ spanned by $ \{ e_{i,i} \}_{i=1}^k$.

We collect a number of facts about these algebras.

\begin{lemma} $A_{m,k} \cong A_{n,k}$ for all $m, n \leq k$ \end{lemma}
\begin{proof} 
The isomorphism is a unitary equivalence induced by elementary row and column operations in $M_k \otimes C[0,1]$.
\end{proof}

We can also see that $A_{j,k} \cong (C[0,1] \otimes M_2) \oplus C[0,1] \oplus C[0,1] \oplus \cdots \oplus C[0,1]$ with $k-2$ copies of $C[0,1]$.  In what follows we will often write proofs as if we are restricting to $j< k$ but all of the results will follow for the case $j=k$ because of the preceding lemma (or by natural changes to the indicated maps).

There is an injective $*$-homomorphism $\lambda_{j,k}: C[0,1] \rightarrow A_{j,k}$ given by 
\[ \lambda_{j,k} (f) = 
\begin{bmatrix} f  & \cdots & 0 & 0 & 0 & 0 & \cdots  & 0 \\ 
\vdots &  \cdots & \ddots & \ddots & \ddots & \ddots & \cdots &  \vdots \\
0  & \cdots & f & 0 & 0 & 0 & \cdots  & 0 \\ 
0  & \cdots & 0 & f & 0 & 0 & \cdots  & 0 \\ 
0  & \cdots & 0 & 0 & f & 0 & \cdots  & 0 \\ 
0  & \cdots & 0 & 0 & 0 & f & \cdots  & 0 \\ 
\vdots  & \cdots & \ddots & \ddots & \ddots & \ddots & \cdots  & \vdots \\
0  & \cdots & 0 & 0 & 0 & 0 & \cdots  & f \end{bmatrix} \] 

Next we let $E_{j,k}: A_{j,k} \rightarrow A_{j,k}$ be given by sending 
\[ \begin{bmatrix} f_{1,1}  & \cdots & 0 & 0 & 0 & 0 & \cdots  & 0 \\ 
0  & \cdots & 0 & 0 & 0 & 0 & \cdots & 0 \\ 
\vdots  & \cdots & \ddots & \ddots & \ddots & \ddots & \cdots &  \vdots \\
0  & \cdots & f_{j-1,j-1} & 0 & 0 & 0 & \cdots  & 0 \\ 
0  & \cdots & 0 & f_{j,j} & f_{j,j+1} & 0 & \cdots  & 0 \\ 
0  & \cdots & 0 & f_{j+1,j} & f_{j+1,j+1}  & 0 & \cdots  & 0 \\ 
0  & \cdots & 0 & 0 & 0 & f_{j+2, j+2} & \cdots  & 0 \\ 
\vdots  & \cdots & \ddots & \ddots & \ddots & \ddots & \cdots &  \vdots \\
0  & \cdots & 0 & 0 & 0 & 0 & \cdots  & 0 \\ 
0  & \cdots & 0 & 0 & 0 & 0 & \cdots  & f_{k,k} \end{bmatrix} \] to the matrix \[ 
\begin{bmatrix} f_{1,1}  & \cdots & 0 & 0 & 0 & 0 & \cdots  & 0 \\ 
0  & \cdots & 0 & 0 & 0 & 0 & \cdots  & 0 \\ 
\vdots  & \cdots & \ddots & \ddots & \ddots & \ddots & \cdots &  \vdots \\
0  & \cdots & f_{1,1} & 0 & 0 & 0 & \cdots  & 0 \\ 
0  & \cdots & 0 & f_{1,1} & 0 & 0 & \cdots  & 0 \\ 
0  & \cdots & 0 & 0 & f_{1,1}  & 0 & \cdots  & 0 \\ 
0  & \cdots & 0 & 0 & 0 & f_{1,1} & \cdots & 0 \\ 
\vdots  & \cdots & \ddots & \ddots & \ddots & \ddots & \cdots &  \vdots \\
0  & \cdots & 0 & 0 & 0 & 0 & \cdots  & 0 \\ 
0  & \cdots & 0 & 0 & 0 & 0 & \cdots  & f_{1,1} \end{bmatrix}. \]

\begin{lemma} The map $ E_{j,k}$ is a conditional expectation onto $\lambda_{j,k}(C[0,1])$ such that $E_{j,k} (D_k) = \lambda_{j,k} \left(\mathbb{C} \cdot 1_{[0,1]}\right)$. \end{lemma}

\begin{proof}
That $E_{j,k}$ is linear and $E_{j,k}^2 = E_{j,k}$ follows from direct computation. By Tomiyama's Theorem \cite[Theorem 1.5.10]{BrownOzawa} we need only verify that $E_{j,k}$ is contractive. So let $X = [f_{m,n}] \in A_{j,k}$ and notice that 
\begin{align*} \left\| E_{j,k} \left( X \right)  \right\|& = \| f_{1,1} (e_{1,1} + e_{2,2} + \cdots + e_{k,k} \| \\ & \leq  
 \| f_{1,1} \| \| e_{1,1} + e_{2,2} + \cdots + e_{k,k} \| \\ 
& = \| f_{1,1} \| \\ & = \| e_{1,1} X e_{1,1} \| \\ 
& \leq \| X \|, \end{align*}
verifying that $E_{j,k}$ is contractive. \end{proof}

Thus we have the following commutative diagrams:
\[ \xymatrix{ A_{j,k} & \ar[l] D_k \\ C[0,1] \ar^{\lambda_{j,k}}[u]  & \ar[l] \ar_{\lambda_{j,k}|_{ \mathbb{C} \cdot 1_{[0,1]}}}[u] \mathbb{C} } \]
of which all of the maps are injections, and
\[ \xymatrix{ C[0,1] \otimes M_k \ar_{E_{j,k}}[d] & \ar[l]\ar^{E_{j,k}|_{D_k}}[d] D_k \\ C[0,1]   & \ar[l] \mathbb{C} } \]
with conditional expectations on the vertical arrows. Applying \cite[Proposition 2.4]{ArmstrongDykemaExelLi} we conclude the following

\begin{proposition}
The induced map $\Lambda_{j,j+1}: C[0,1] \free{\mathbb{C}} C[0,1] \rightarrow A_{j,k} \free{D_k} A_{j+1,k}$ is injective.
\end{proposition}

We will be focused on free products of multiple algebras at once, so we will need to extend the result of \cite{ArmstrongDykemaExelLi} to apply to multiple free products simultaneously. To do that we turn to a version of Boca's Theorem \cite{Boca} which was extended in \cite{DavidsonKakariadis}. Essentially it allows us to show that the ``free product'' of a finite number of conditional expectations is again a conditional expectation, given some additional hypotheses on the algebras. We can then use induction to extend \cite[Proposition 2.4]{ArmstrongDykemaExelLi}.

We refer the reader to the introduction of \cite{DavidsonKakariadis} for the relevant details on the construction of amalgamated free products. For our purposes, given unital C$^*$-algebras $C_i$ with common C$^*$-subalgebra $C$ the amalgamated free product $C_1 \free{C} C_2 \free{C} \cdots \free{C} C_n$ contains a dense subalgebra spanned by elements of the form $c_{i_1} c_{i_2} \cdots c_{i_m}$ where $c_{i_j}$ is in one of the $C_i$'s, but no two consecutive terms $c_{i_j}$ and $c_{i_{j+1}}$ are contained in the same $C_i$. Thus a continuous linear map with domain $C_1 \free{C} C_2 \free{C} \cdots \free{C} C_n$ is uniquely determined by its action on these generating elements. Now, when there exist conditional expectations $E_i: C_i \rightarrow C$, we can say more.  In this case the algebra $C_1 \free{C} C_2 \free{C} \cdots \free{C} C_n$ is the closure of \[ C \oplus \left( \oplus_{ 1 \leq p \leq \infty, i_k \neq i_{k+1}} \ker E_{i_1} \otimes \ker E_{i_2} \otimes \ker E_{i,p} \right). \] This allows \cite{DavidsonKakariadis} a stronger version of Boca's theorem which we will utilize in the following result. 

\begin{proposition} Consider a finite family of unital C$^*$-algebras with inclusions $ B_i \subseteq A_i$ and assume that there are common C$^*$-subalgebras $C' \subseteq B_i$ for all $i$, $C \subseteq A_i$ for all $i$ and $C' \subseteq C$. Assume for each $i$ there is a surjective conditional expectation $F_i: A_i \rightarrow C$. 
%and there is an inclusion $\lambda: B_1 \free{C'} B_2 \free{C'} \cdots \free{C'} B_k \rightarrow A_1 \free{C} A_2 \free{C} \cdots \free{C} A_k$. 
If there are conditional expectations $E_i: A_i \rightarrow B_i$ with $E_i(c) = E_j(c)$ for all $ c \in C$, and $E_i|_{C}: C \rightarrow C'$ is a surjective conditional expectation for each $i$, then there is a surjective conditional expectation \[ G: A_1 \free{C} A_2 \free{C} \cdots \free{C} A_k \rightarrow B_1 \free{C'} B_2 \free{C'} \cdots \free{C'} B_k \] such that $E_i = G|_{A_i}$ for all $i$.  \end{proposition}

\begin{proof}
Since the maps $E_i$ are onto conditional expectations, they are completely positive and are the identity on their range which is $B_i$. Then using the natural inclusion maps $\iota_i: B_i \rightarrow B_1 \free{C'} B_2 \free{C'} \cdots \free{C'} B_n$, we define \[ G_i = \iota_i \circ E_i: A_i \rightarrow B_1 \free{C'} B_2 \free{C'} \cdots \free{C'} B_k \] and notice that by hypothesis $G_i|_C = G_j|_C$ for all $i,j$. 

Now let $\mathcal{H}$ be a Hilbert space so that there exists \[ \pi: B_1 \free{C'} B_2 \free{C'} \cdots \free{C'} B_k \rightarrow B(\mathcal{H}) \] which is a faithful $*$-representation. Then $ \pi \circ G_i$ is a unital completely positive map from $A_i \rightarrow B(\mathcal{H})$ such that $\pi \circ G_i|_C = \pi \circ G_j|_C$ is a $*$-representation for all $i, j$. It follows from \cite[Theorem 3.4]{DavidsonKakariadis} that there is a unital completely positive map $\Phi: A_1 \free{C} A_2 \free{C} \cdots \free{C} A_n \rightarrow B(\mathcal{H})$ such that $\Phi|_{A_i} = \pi \circ G_i$ and \[ \Phi(a_{j_1} a_{j_2} \cdots a_{j_k}) =  \pi \circ G_{j_1}(a_{j_1})\pi \circ G_{j_2})( a_{j_2}) \cdots \pi \circ G_{j_k}(a_{j_k}) \] when $ j_{i} \neq j_{i+1} $ and $ a_{j_i} \in \ker I_{j_i}$.

In particular, the existence of the conditional expectations $F_i: A_i \rightarrow C$ means that $A$ is the closed linear span of the space \[ C \oplus \left( \oplus_{m \geq 1} \ker E_{i_1} \otimes \cdots \otimes \ker E_{i_m} \right)\] and $\ker E_{i_j} \subseteq \ker \pi \circ G_{i_j}$. Applying the formula for $\Phi$ to elements of the form $b_0 b_{i_1} b_{i_2} \cdots b_{i_m}$ with $b_0 \in C'$, and $ b_{i_j} \in \ker F_j|_{B_j} $ and $i_{j} \neq i_{j+1}$ establishes that $\Phi$ is a surjective unital completely positive map onto $ \ran \pi$. It follows that $ G = \pi^{-1} \circ \Phi$ is a unital completely positive map from $A_1 \free{C} A_2 \free{C} \cdots \free{C} A_k \rightarrow B_1 \free{C'} B_2 \free{C'} \cdots \free{C'} B_k$. As it is also an idempotent we can conclude that $G$ is a conditional expectation. \end{proof}

In the absence of the conditional expectation $F_i$, we cannot establish the additional detail that the range of our extended map into $B(\mathcal{H}) $ has range inside the range of $\pi$ so that the argument breaks down. However for the algebras we're interested in there is a natural conditional expectation $F_{j,k}: A_{j,k} \rightarrow D_k$. given by \[ F_{j,k} \left( \sum_{r,s=1}^ke_{r,s} \otimes f_{r,s} \right) = \sum_{r=1}^k f_{r,r}(0)e_{r,r}, \] which allows us to apply the following extension of \cite[Proposition 2.4]{ArmstrongDykemaExelLi}.

\begin{theorem}
Consider a finite family of C$^*$-algebras with inclusions $ B_i \subseteq A_i$ and assume that there are common C$^*$-subalgebras $C' \subseteq B_i$ for all $i$, $C \subseteq A_i$ for all $i$ and $C' \subseteq C$. If there are conditional expectations $F_i: A_i \rightarrow C$ and $E_i: A_i \rightarrow B_i$ , the latter of which when restricted to $C$ is a conditional expectation onto $C'$ then the natural map \[ \lambda: B_1 \free{C'} B_2 \free{C'} \cdots \free{C'} B_n \rightarrow A_1 \free{C} A_2 \free{C} \cdots \free{C} A_n \] is injective.
\end{theorem}

\begin{proof} The base case is \cite[Proposition 2.4]{ArmstrongDykemaExelLi} so we assume that it is true for $ 1 \leq k < n$ and consider the case of $k=n$. Then by the previous proposition there is a conditional expectation $E_{n-1}: A_1\free{C} A_2 \free{C} \cdots \free{C} A_{n-1} \rightarrow B_1 \free{C'} B_2 \free{C'} \cdots \free{C'} B_{n-1}$ which when restricted to $C$ is a conditional expectation onto $C'$. Then we apply the base case result to the following commutative diagram of C$^*$-algebras
\[ \xymatrix{ B_1 \free{C'} B_2 \free{C'} \cdots \free{C'} B_{n-1} \ar[d] & C' \ar[l] \ar[r] \ar[d] & B_n \ar[d] \\ A_1 \free{C} A_2 \free{C} \cdots \free{C} A_{n-1} & C \ar[l] \ar[r] & A_n } \] to get that the map $\lambda$ is indeed injective.
\end{proof}

Applying this to the free products we're considering the following is a direct application of the preceding.

\begin{corollary} There is a natural injection of the $j$-fold free product \[ C[0,1] \free{\mathbb{C}} C[0,1] \free{\mathbb{C}} \cdots \free{\mathbb{C}} C[0,1]\] into the algebra \[ A_{1,k} \free{D_k} A_{2, k} \free{D_k} \cdots \free{D_k} A_{j,k} \] for all $1 \leq j \leq k$.\end{corollary}

We now consider the other main inclusion of free products that will play a result in our results. We start again with introducing some notation.
We let $B_{j,k}$ denote the subalgebra of $A_{j,k}$ of the form 
\[ \begin{bmatrix} 
\mathbb{C}  & \cdots & 0 & 0 & 0 & 0 & \cdots  & 0 \\ 
\vdots &  \cdots & \ddots & \ddots & \ddots & \ddots & \cdots  & \vdots \\
0  & \cdots & \mathbb{C} & 0 & 0 & 0 & \cdots  & 0 \\ 
0  & \cdots & 0 & \mathbb{C} & \mathbb{C} & 0 & \cdots  & 0 \\ 
0  & \cdots & 0 & \mathbb{C} & \mathbb{C} & 0 & \cdots  & 0 \\ 
0  & \cdots & 0 & 0 & 0 & \mathbb{C} & \cdots  & 0 \\ 
\vdots  & \cdots & \ddots & \ddots & \ddots & \ddots & \cdots  & \vdots \\
0  & \cdots & 0 & 0 & 0 & 0 & \cdots  & \mathbb{C} \end{bmatrix}. \]

Then we can see that as before $B_{m,k} \cong B_{n,k}$ for $m, n \leq k-1$ and that the algebra $B_{m,k} \cong M_2 \oplus \mathbb{C}^{k-2}$. In this context we do not need conditional expectations to get at our result.  In fact we have the following commutative diagram of inclusions of C$^*$-algebras
\[ \xymatrix{  A_{m,k} &  \\ & D_k \ar[ul] \ar[dl] \\ B_{m,k} \ar[uu] }. \] It is a straightforward induction argument using \cite[Theorem 4.2]{Pedersen} that yields the following theorem. 

\begin{theorem}
The induced map \[ \sigma: B_{1,k} \free{D_k} B_{2,k} \free{D_k} \cdots \free{D_k} B_{l,k} \rightarrow A_{1,k} \free{D_k} A_{2,k} \free{D_k} \cdots \free{D_k} B_{l,k} \] is injective for $ 2 \leq l \leq k$.
\end{theorem}

Now we can see from \cite{Duncangraphs} and noting that $B_{i,k}$ is the graph C$^*$-algebra corresponding to $k$ vertices and a single edge connecting the $(i-1)$-vertex to the $i$-th vertex to see that \[ B_{1,k} \free{D_k} B_{2,k} \free{D_k} \cdots \free{D_k} B_{k-1,k} \cong M_k \] and that \[ B_{1,k} \free{D_k} B_{2,k} \free{D_k} \cdots \free{D_k} B_{k,k} \] is the graph algebra associated to a cycle of length $k$ and hence \cite{Raeburn} is isomorphic to $M_k \otimes C(\mathbb{T})$.

To understand the free products \[ A_{1,k} \free{D_k} \cdots \free{D_k} A_{k-1,k} \] and \[ A_{1,k} \free{D_k} \cdots \free{D_k} A_{k,k} \] we will first notice that the range of the two injections $\pi$ and $\sigma$ together include generating sets for the free product. Hence we can consider how elements in the ranges of these two injections ``multiply'' together to get a sense of the entire free product algebra. In particular, we know the structure of the ranges of $\pi$ (free products of copies of $C[0,1]$) and $\sigma$ (full matrix algebras or full matrix algebras tensored with $C(\mathbb{T})$). The interesting piece is what happens when we consider elements of the range of $\pi$ multiplied against elements of the range of $\sigma$. Because we are going to investigate this carefully we will consider the first algebra and then utilize that in understanding the second algebra.

\begin{theorem} The algebra $A_{1,k} \free{D_k} A_{2,k} \free{D_k} \cdots \free{D_k} A_{k-1,k}$ is isomorphic to \[ \left( C[0,1] \free{\mathbb{C}} C[0,1] \free{\mathbb{C}} \cdots \free{\mathbb{C}} C[0,1] \right) \otimes M_k \] where there are $k-1$ copies of $C[0,1]$ in the free product. \end{theorem}

\begin{proof}

The range of $\sigma$ creates a set of matrix units inside \[ \mathcal{A}:= A_{1,k} \free{D_k} \cdots \free{D_k} A_{k-1,k}\] which we will denote by $F_{i,j}$. We know that the range of $\sigma$ is the linear span of the $F_{i,j}$s, that \[ F_{i,j}F_{m,n} = \begin{cases} F_{i,n} & m =n \\ 0 & \mbox{ otherwise} \end{cases}, \]that $ 1_{\mathcal{A}} = F_{1,1} + F_{2,2} + \cdots +  F_{k,k}$, and $D_k $ is the span of $\{ F_{1,1}, F_{2,2}, \cdots, F_{k,k} \}$. More importantly since amalgamation occurs over $D_k$ and as $A_{i,k}$ is a direct sum and a tensor product, from the universal property of tensor products, we know that inside $A_{i,k}$ the image of $f \in C[0,1]$ will commute with all of the $F_{i,i}$ and \[ f = \sum_{i=}^k F_{i,i} f F_{i,i}. \] Since this is true for each $ f \in C[0,1]$ it must be true for any $F$ in the range of $\pi$, and hence the range of $\pi$ must be ``diagonal'' with respect to the matrix units for $M_n$ inside the range of $\sigma$. Then for a generic multiplication $F_{i,j} F = F_{i,j} F_{j,j} f F_{j,j}$ and $ FF_{i,j} = F_{i,i} F F_{i,i} F_{i,j}$ and hence the algebra $\mathcal{A}$ is generated inside a matrix algebra over $C[0,1] \free{\mathbb{C}} C[0,1] \free{\mathbb{C}} \cdots \free{\mathbb{C}} C[0,1]$. That it is onto requires only establishing that there is a representation of the free product algebra onto the matrix algebra, which follows from the universal property of the free product with the fact that each of the $B_{i,k}$s has a natural inclusion into the matrix algebra which is onto a generating set of that matrix algebra.
\end{proof}

We can now use a result of Pedersen \cite[Theorem 5.5]{Pedersen} to get an analogous result in the remaining case.
 
\begin{theorem}\label{diagram} The algebra $A_{1,k} \free{D_k} A_{2,k} \free{D_k} \cdots \free{D_k} A_{k,k}$ is isomorphic to  \[\left( \left( C[0,1] \free{\mathbb{C}} C[0,1] \free{\mathbb{C}} \cdots \free{\mathbb{C}} C[0,1] \right)\free{\mathbb{C}} C(\mathbb{T}) \right)\otimes M_k \] where there are $k$ copies of $C[0,1]$ in the free product. \end{theorem}

\begin{proof}Consider the commutative pushout diagram
\[ 
\xymatrix{ \mathbb{C} \ar[d] \ar[r] & C(\mathbb{T}) \ar[d] \\ 
C[0,1] \free{\mathbb{C}} C[0,1] \free{\mathbb{C}} \cdots \free{\mathbb{C}} C[0,1] \ar[r] & \left( C[0,1] \free{\mathbb{C}} C[0,1] \free{\mathbb{C}} \cdots \free{\mathbb{C}} C[0,1] \right) \free{\mathbb{C}} C( \mathbb{T}) 
} 
\]
and notice that the maps are proper and hence \cite[Theorem 5.5]{Pedersen} tells us that the following commutative diagram is a pushout 
\[ 
\xymatrix{ \mathbb{C}\otimes M_k \ar[d] \ar[r] & C(\mathbb{T}) \otimes M_k \ar[d] \\ 
C[0,1] \free{\mathbb{C}} \cdots \free{\mathbb{C}} C[0,1]  \otimes M_k \ar[r] & \left( C[0,1]  \free{\mathbb{C}} \cdots \free{\mathbb{C}} C[0,1] \right) \free{\mathbb{C}} C( \mathbb{T}) \otimes M_k.} 
\]
In other words \[ \left(C[0,1] \free{\mathbb{C}} C[0,1] \free{\mathbb{C}} \cdots \free{\mathbb{C}} C[0,1]  \otimes M_k \right) \free{\mathbb{C} \otimes M_k} \left( C(\mathbb{T}) \otimes M_k \right)\] is isomorphic to \[ \left(C[0,1] \free{\mathbb{C}} C[0,1] \free{\mathbb{C}} \cdots \free{\mathbb{C}} C[0,1] \free{\mathbb{C}} C( \mathbb{T}) \right) \otimes M_k.\]

We also have a third pushout diagram which puts the algebras we just recognized in the bottom left and top right corner and the algebra we wish to describe into the bottom right hand corner:
\[
\xymatrix{  B_{1,1} \free{D_k} A_{2,k} \free{D_k} \cdots \free{D_k} A_{k-1,k} \free{D_k} B_{k,k} \ar[r] \ar[d] & B_{1,k} \free{D_k} A_{2,k} \free{D_k} \cdots \free{D_k} A_{k,k} \ar[d] \\ A_{1,k} \free{D_k} A_{2,k} \free{D_k} \cdots \free{D_k} A_{k-1,k} \free{D_k} B_{k,k} \ar[r] &  A_{1,k} \free{D_k} A_{2,k} \free{D_k} \cdots \free{D_k} A_{k,k}. }
\]

The algebra in the bottom left corner of the second pushout diagram is isomorphic to both \[A_{1,k} \free{D_k} A_{2,k} \free{D_k} \cdots \free{D_k} A_{k-1,k} \] and \[ A_{2,k} \free{D_k} A_{3,k} \free{D_k} \cdots \free{D_k} A_{k,k}. \] The algebra in the upper right corner of the same pushout diagram is isomorphic to $B_{1,k} \free{D_k} B_{2,k} \free{D_K} \cdots \free{D_K} B_{k,k}$. Alternatively, we can recognize these two algebras as, respectively \[ \left( A_{1,k} \free{D_k} \cdots \free{D_k} A_{k-1,k} \free{D_k} D_k \right) \free{B_{1,1} \free{D_k} \cdots \free{D_k} B_{k-1,k-1} \free{D_k} D_k} \left( B_{1,1} \free{D_k} \cdots \free{D_k} B_{k,k} \right)\] and 
\[ \left( D_k \free{D_k} A_{2,k} \free{D_k} \cdots \free{D_k} A_{k,k} \right) \free{D_k \free{D_k} B_{2,2} \free{D_k} \cdots \free{D_k} B_{k,k}} \left( B_1 \free{D_k} \cdots \free{D_k} B_k \right),\] which can be rewritten as \[ B_{1,k} \free{D_k} A_{2,k} \free{D_k} A_{3,k} \free{D_k} \cdots \free{D_k} A_{k,k}\] and \[ A_{1,k} \free{D_k} A_{2,k} \free{D_k} \cdots \free{D_k} A_{k-1,k} \free{D_k} B_{k,k}.\]

So we can use the previous theorem to turn the third pushout diagram into the following:
\[ 
\xymatrix{ B_{1,1} \free{D_k} A_{2,k} \free{D_k} \cdots \free{D_k} A_{k-1,k} \free{D_k} B_{k,k} \ar[r] \ar[d] & M_k \left( C[0,1] \free{\mathbb{C}} \cdots \free{\mathbb{C}} C[0,1] \free{\mathbb{C}} C(\mathbb{T}) \right) \ar[d] \\ M_k\left( C[0,1] \free{\mathbb{C}} \cdots \free{\mathbb{C}} C[0,1] \free{\mathbb{C}} C(\mathbb{T}) \right)  \ar[r] & A_{1,k} \free{D_k} A_{2,k} \free{D_k} \cdots \free{D_k} A_{k,k}. }
\]

Where there are $k-1$ copies of $C[0,1]$ in both the upper right and bottom left corners of the commutative diagram. We need only consider how \[ B_{1,1} \free{D_k} A_{2,k} \free{D_k} A_{3,k} \free{D_k} \cdots \free{D_k} A_{k-1,k} \free{D_k} B_{k,k} \] embeds into the two different corners of the commutative diagram. Looking at the previous theorem on the generators of this subalgebra we can follow the maps to see that for the top row the embedding is into everything but the first copy of $ C[0,1]$ and the column embedding is into everything but the last copy of $C[0,1]$. In particular, \[ B_{1,1} \free{D_k} A_{2,k} \free{D_k} A_{3,k} \free{D_k} \cdots \free{D_k} A_{k-1,k} \free{D_k} B_{k,k} \] is isomorphic to \[ \left( C[0,1] \free{\mathbb{C}} C[0,1] \free{\mathbb{C}} \cdots \free{\mathbb{C}} C[0,1] \free{\mathbb{C}} C(\mathbb{T}) \right) \otimes M_k, \] with $k-2$ copies of $ C[0,1]$.  The description of \[ A_{1,k} \free{D_k} A_{2,k} \free{D_k} \free{D_k} \cdots \free{D_k} A_{k,k} \] now follows.

\end{proof}

We can now consider these results in the context of determining the algebras C$^*_{\rm max} (T_n)$ and C$^*_{\rm max} (B_n)$. 

\section{Some maximal covers}

We start with a result of Blecher \cite{Blecher} which gives the maximal C$^*$-cover of the upper triangular $2 \times 2$ matrix algebra. 

\begin{proposition} We denote by $ \omega$ the function $\sqrt{1-t} \in C[0,1]$ then the inclusion \[ \begin{bmatrix} \lambda_{1,1} & \lambda_{1,2} \\ 0 & \lambda_{2,2} \end{bmatrix} \mapsto \begin{bmatrix} \lambda_{1,1} & \lambda_{1,2} \omega \\ 0 & \lambda_{2,2} \end{bmatrix} \] induces a completely contractive map which generates  \[ {\rm C}^*_{\rm max} (T_2) = \begin{bmatrix} C[0,1] & \omega C[0,1] \\ \omega C[0,1] & C_0[0,1] \end{bmatrix}. \]
\end{proposition} 

The set $\omega C[0,1]$ is a principal ideal which can be identified as those functions in $C[0,1]$ which are $0$ when evaluated at $1$. A key element of our analysis is that C$^*_{\rm max} (T_2)$ can be generated by matrices of the form \[ \begin{bmatrix} f_{1,1}(t) & 0 \\ 0 & 0 \end{bmatrix}, \begin{bmatrix} 0 & 0 \\ 0 & f_{2,2}(t) \end{bmatrix}, \mbox{ and } \begin{bmatrix} 0 & \omega \\ 0 & 0 \end{bmatrix} \] where $f_{i,i} \in C[0,1]$. Combining the previous result with another result of Blecher \cite{Blecher} we can find the maximal C$^*$-covers of upper triangular $n \times n$ matrix algebras and certain semicrossed products studied by Peters in \cite{Peters}.  To do this we recognize that these algebras arise as graph operator algebras \cite{Kribs} corresponding to the directed graphs $L_n$ and $C_n$, respectively, where $L_n$ is the directed graph with $n$ vertices $v_1, v_2 ,\cdots, v_n$ and $n-1$ edges $e_1, e_2, \cdots, e_{n-1} $ where $s(e_i) = v_i$ and $r(e_i) = v_{i+1}$ and $C_n$ is the $n$-cycle graph. The key point that we use is from \cite{Duncangraphs} which applies since each edge has unique range. We can summarize the graph algebra results with the following restatements of \cite[Theorem 4]{Duncangraphs} in this specific context.

\begin{theorem} Let $A_1= T_2 \oplus \mathbb{C}^{n-2}, A_2 = \mathbb{C} \oplus T_2 \oplus \mathbb{C}^{n-3}, \cdots A_{n-1} = \mathbb{C}^{n-2} \oplus T_2$ viewing each as subalgebras of $M_n$, then \[ T_n = A(L_n) = A_1 \free{D_n} A_2 \free{D_n} \cdots \free{D_n} A_{n-1}, \] where amalgamation occurs over the subalgebra of $n \times n$ diagonal matrices.  If we further write $A_n$ for the matrix algebra \[ \begin{bmatrix} \lambda_{1,1} & 0 & 0 & \cdots & 0 & 0 \\ 0 & \lambda_{2,2} & 0 & \cdots & 0  & 0 \\ \vdots & \vdots & \ddots & \cdots & \vdots & \vdots \\ \lambda_{n,1} & 0 & 0 & \cdots & 0 & \lambda_{n,n} \end{bmatrix} \subseteq M_n \] then \[ A(C_n) = A_1 \free{D_n} A_2 \free{D_n} \cdots \free{D_n} A_{n-1} \free{D_n} A_n. \]
\end{theorem}

We know from combining universal properties that maximal C$^*$-covers of free products, amalgamated over common C$^*$-subalgebras, are the free products of maximal C$^*$-covers amalgamated over the same C$^*$-subalgebra \cite[Proposition 2.2]{Blecher}. Combining this fact with the previous theorem yields the following. 

\begin{theorem} For $ n \geq 3$ \[ {\rm C}^*_{\rm max} (T_n) =  {\rm C}^*_{\rm max} (A_1) \free{D_n} {\rm C}^*_{\rm max} (A_2) \free{D_n} \cdots \free{D_n} {\rm C}^*_{\rm max} (A_{n-1}) \] and for $ n \geq 2$ \[ {\rm C}^*_{\rm max} (A(C_n)) = {\rm C}^*_{\rm max} (A_1) \free{D_n} {\rm C}^*_{\rm max} (A_2) \free{D_n} \cdots \free{D_n} {\rm C}^*_{\rm max} (A_{n-1}) \free{D_n} {\rm C}^*_{\rm max} (A_n) \]
\end{theorem}

It is a straightforward application of the fact the maximal C$^*$-cover of a C$^*$-algebra is the C$^*$-algebra and that maximal C$^*$-covers of direct sums are direct sums of the C$^*$-covers to see that for $ 1 \leq i \leq n-1$, 
C$^*_{\rm max} (A_i) = \mathbb{C}^{i-1} \oplus C^*_{\rm max} (T_2) \oplus \mathbb{C}^{n-(i+1)}$ and 
\[ {\rm C}^*_{\rm max} (A_n) = \begin{bmatrix} C[0,1] & 0 & 0 & \cdots & 0 & \omega C[0,1] \\ 0 & 0 & 0 & \cdots & 0  & 0 \\ \vdots & \vdots & \ddots & \cdots & \vdots & \vdots \\ \omega C[0,1] & 0 & 0 & \cdots & 0 & C[0,1] \end{bmatrix}. \] 
Viewing each of these algebras as subalgebras of $M_n \otimes C[0,1]$ we see how the algebras from the previous section come into play. 

\begin{proposition}
The maximal C$^*$-cover of $T_n$ and $A(C_k)$ are subalgebras of \[ \left(C[0,1] \free{\mathbb{C}} C[0,1] \free{\mathbb{C}} \cdots \free{\mathbb{C}} C[0,1] \right) \otimes M_n \] and \[ \left( C[0,1] \free{\mathbb{C}} C[0,1] \free{\mathbb{C}} \cdots \free{\mathbb{C}} C[0,1] \free{\mathbb{C}} C( \mathbb{T}) \right) \otimes M_k,\] respectively.
\end{proposition}

\begin{proof}
Since all of the algebras C$^*_{\rm max} (A_i)$ contain the diagonal algebras $D_n$ and amalgamation for the maximal C$^*$-algebras and the algebra $A$ occur over the same C$^*$-subalgebra, $D_n$, we can apply the result of Pedersen \cite[Theorem 5.5]{Pedersen} directly. 
\end{proof}

In theory, this gives us a practical method of computing the maximal C$^*$-covers for these algebras. We just determine the subalgebra generated by the generating subsets of the algebras that make up the free product. However this analysis is not necessarily amenable to a nice closed form solution. We will start by illustrating this result in the context of $T_3$ where the computations are manageable. We are dealing with multiple copies of $C[0,1]$ so we will introduce some notation to keep track of where algebra elements lie.

Considering $C[0,1] \free{\mathbb{C}} C[0,1]$ the two algebras separately are continuous functions in $t$ and hence the free product can be viewed as functions in two noncommuting self-adjoint variables $t_1$ and $t_2$. In general, we will write $C\langle t_1, t_2, \cdots, t_n \rangle$ as shorthand for $C[0,1] \free{\mathbb{C}} C[0,1] \free{\mathbb{C}} \cdots \free{\mathbb{C}} C[0,1]$ so that we can track which copy of $C[0,1]$ a function might appear in. Similarly, we will write $\omega_i$ to represent $\sqrt{1-t_i}$ in the $i$-th copy of $C[0,1]$, since it plays an important role in the maximal C$^*$-cover.

For a collection of elements $\{ f_1, f_2, \cdots, f_j \} \in C\langle t_1, t_2, \cdots, t_n \rangle$ we write 
\begin{itemize} \item $(f_i) C\langle t_1, t_2, \cdots, t_n \rangle$ to denote the closed subspace consisting of all elements of the form $f_i A$ where $A \in C\langle t_1, t_2, \cdots, t_n \rangle$,
\item $C\langle t_1, t_2, \cdots, t_n \rangle (f_i)$ to denote the closed subspace consisting of all elements of the form $A f_i$ where $A \in C\langle t_1, t_2, \cdots, t_n \rangle$,
\item and, $(f_i) C \langle t_1, t_2, \cdots, t_n \rangle (f_j)$ to denote the closed subspace consisting of all elements of the form $f_i A f_j$ where $A \in C\langle t_1, t_2, \cdots, t_n \rangle$. Note here that $i$ may equal $j$.
\end{itemize}

Having established the notation we provide the following characterizations of the maximal C$^*$-cover for $T_3$.

\begin{proposition} \label{max} \[ {\rm C}^*_{\rm max} (T_3) = \begin{bmatrix} C\langle t_1 \rangle + \omega_1 C\langle t_1, t_2 \rangle \omega_1 & \omega_1 C\langle t_1, t_2 \rangle & w_1 C \langle t_1, t_2 \rangle \omega_2 \\ C\langle t_1, t_2, \rangle \omega_1 & C \langle t_1, t_2 \rangle & C \langle t_1, t_2 \rangle \omega_2 \\ \omega_2 C\langle t_1, t_2 \rangle \omega_1 & \omega_2 C \langle t_1, t_2 \rangle & C \langle t_2 \rangle + \omega_2 C \langle t_1, t_2 \rangle \omega_2  \end{bmatrix}.  \]
\end{proposition}

\begin{proof} 
We will denote by $\mathfrak{A}$ the subset of $M_3 \otimes C\langle t_1, t_2, t_3 \rangle$ given by \[ \begin{bmatrix} C\langle t_1 \rangle + \omega_1 C\langle t_1, t_2 \rangle \omega_1 & \omega_1 C\langle t_1, t_2 \rangle & w_1 C \langle t_1, t_2 \rangle \omega_2 \\ C\langle t_1, t_2, \rangle \omega_1 & C \langle t_1, t_2 \rangle & C \langle t_1, t_2 \rangle \omega_2 \\ \omega_2 C\langle t_1, t_2 \rangle \omega_1 & \omega_2 C \langle t_1, t_2 \rangle & C \langle t_2 \rangle + \omega_2 C \langle t_1, t_2 \rangle \omega_2  \end{bmatrix}.  \]

Straightforward calculations will verify that $\mathfrak{A}$ is a linear subspace of $M_3 \otimes C\langle t_1, t_2 \rangle$. Similarly one can use matrix multiplication to see that it is, in fact, a C$^*$-subalgebra.  

Now the inclusion
\[ 
\begin{bmatrix} \lambda_{1,1} & \lambda_{1,2} & \lambda_{1,3} \\ 0 & \lambda_{2,2} & \lambda_{2,3} \\ 0 & 0 & \lambda_{3,3} \end{bmatrix} \mapsto \begin{bmatrix} \lambda_{1,1} & \lambda_{1,2}\omega_1 & \lambda_{3, 1}\omega_{1}\omega_2 \\ 0 & \lambda_{2,2} & \lambda_{2,3}\omega_2 \\ 0 & 0 & \lambda_{3,3} \end{bmatrix}
\]
gives an inclusion of $T_3 \subseteq \mathfrak{A}$. The proof then comes down to verifying that every element of $\mathfrak{A}$ is contained in the algebra generated by the range of these two inclusions. 

Since we know that $T_2$ embeds into $T_3$ as $2\times 2$ matrix subalgebras in two ways we have that C$^*_{\rm max} (T_2)$ embeds into $M_3 \otimes C\langle t_1, t_2 \rangle$ and in the same way in particular we have that it embeds in $\mathfrak{A}$ in similar matrix constructions and hence the range of $T_3$ must generate, at minimum the following subsets of $ \mathfrak{A}$: \[ \begin{bmatrix} C\langle t_1 \rangle & \omega_1 & 0 \\ 0 & C\langle t_1 \rangle & 0 \\ 0 & 0 & 0 \end{bmatrix} \mbox{ and } \begin{bmatrix} 0 & 0 & 0 \\ 0 & C\langle t_2 \rangle & \omega_2 \\ 0 & 0 & C \langle t_2 \rangle  \end{bmatrix}. \]
Considering adjoints and focusing on the multiplication in the $2$-$2$ entry we can generate the subset of $\mathfrak{A}$ of the form \[ \begin{bmatrix} C\langle t_1 \rangle & \lambda_{1,2} \omega_1 & 0  \\ \lambda_{2,1} \omega_1 & C \langle t_1, t_2 \rangle & \lambda_{2, 3} \omega_2 \\ 0 & \lambda_{3,2} \omega_2 & C \langle t_2 \rangle \end{bmatrix} \] where the $ \lambda_{1,j}$ are complex scalars. Considering the multiplication of two elements of that form and linear combinations we get matrices of the form \[ \begin{bmatrix} f + \omega_1 g_{1,1} \omega_1 & \omega_1 g_{1,2} & \omega_1 g_{1,3} \omega_2 \\ g_{2,1} \omega_1 & g_{2,2} & g_{2,3} \omega_2 \\ \omega_2 g_{3,1} \omega_1 & \omega_2 g_{3,2} & h + \omega_2 g_{3,3} \omega_2 \end{bmatrix} \] where $f \in C\langle t_1 \rangle$, $h \in C\langle t_2 \rangle$ and the $g_{i,j} \in C \langle t_1, t_2 \rangle$.  Since this describes every element of the closed subalgebra $\mathfrak{A}$ we have the intended result.

\end{proof}

Since $C[0,1] \free{\mathbb{C} } C[0,1]$ is not exact by \cite{Wasserman} we already, in this ``simple'' nonselfadjoint operator algebra, see that the maximal C$^*$-cover is not exact (this is similar in spirit to the ``simple'' case of $M_2(\mathbb{C})$ in \cite{KirchbergWassermann}).

\begin{corollary} For $ n \geq 2$, ${\rm C}^*_{\rm max} (T_{n})$ is not exact and  ${\rm C}^*_{\rm max} (A(C_n))$ is not exact. 
\end{corollary}

\begin{proof} By the previous proposition we know that this result is true for $T_3$. In general, though we can see that for $ n \geq 3$ the $2$-$2$ entry of the maximal C$^*$-cover of $T_n$ will, at minimum, contain a copy of $C \langle t_1, t_2 \rangle$ and hence they are all non-exact when $ n \geq 2$. The same subalgebra will be in the $2$-$2$ corner of C$^*_{\rm max} (A(C_n))$ giving us non-exactness for these algebras as well. 
\end{proof}

In addition to this pathology of the maximal C$^*$-covers, the algebraic calculations for bigger values of $n$ quickly become overwhelming. Considering even the case of $T_4$ begins to illustrate the difficulty of generating a general expression for $T_n$.

\begin{example}
We focus on the $1$-$4$ entry in C$^*_{\rm max} (T_4)$. Considering multiplications of generating elements as in the proof of the previous proposition, we can get at least all the elements of the form $\omega_1 C \langle t_1, t_2 \rangle \omega_2 C \langle t_2, t_3 \rangle \omega_3$ but the set of such elements is not a subspace. Hence one must consider the closed span of all elements of that form in the $1$-$4$ entry. It quickly becomes unclear what a general expression would look like, and those problems only compound as the number of variables increases. 
\end{example}

Finally, we will consider the graph C$^*$-algebra for the $2$-cycle graph, although we will not get a complete answer here. We will present a plausible candidate for the maximal C$^*$-algebra although as before generalizing to longer cycle-graphs is not obvious or straightforward how to proceed. 

\begin{example} 

To determine this algebra we recognize it as the C$^*$-subalgebra of $A_{1,2} \free{D_2} A_{2,2}$ generated by the matrices (abusing notation) 
$\begin{bmatrix} 0 & \omega_1 \\ 0 & 0 \end{bmatrix}$, $\begin{bmatrix} 0 & \omega_2 \\ 0 & 0 \end{bmatrix}$ and $ \begin{bmatrix} f\langle t_1, t_2 \rangle & 0 \\ 0 & g \langle t_1, t_2 \rangle \end{bmatrix} $ inside $C[0,1] \free{\mathbb{C}} C[0,1] \free{\mathbb{C}} C(\mathbb{T}) \otimes M_2$. This looks like, as in the previous example, matrices of the form \[ \begin{bmatrix} C\langle t_1, t_2 \rangle & \begin{matrix} C\langle t_1, t_2 \rangle \omega_1 C \langle t_1, t_2 \rangle \\ + C\langle t_1, t_2 \rangle \omega_2 C\langle t_1, t_2 \rangle \end{matrix} \\ \begin{matrix} C\langle t_1, t_2 \rangle \omega_1 C \langle t_1, t_2 \rangle \\ + C\langle t_1, t_2 \rangle \omega_2 C\langle t_1, t_2  \rangle \end{matrix} & C\langle t_1, t_2 \rangle \end{bmatrix}. \]

Notice that the copy of $C(\mathbb{T})$ is generated by $M_2 \free{D_2} M_2$ inside $A_{1,2} \free{D_2} A_{2,2}$ the generators of which are not inside this subalgebra. 
\end{example}

\bibliographystyle{amsplain}

\end{document}